\documentclass[reqno]{amsproc}
\usepackage{amsfonts}
\usepackage{graphicx}
\usepackage{amscd}
\usepackage{amsmath}
\usepackage{amssymb}
\usepackage{latexsym}
\usepackage{hyperref}
\usepackage[all]{xy}
\usepackage{color}
\usepackage{mathrsfs}
\usepackage[latin1]{inputenc}
\theoremstyle{plain}

\newtheorem{corollary}{\bf Corollary}

\newtheorem{lemma}{\bf Lemma}

\newtheorem{proposition}{\bf Proposition}
\newtheorem{remark}{Remark}

\newtheorem{theorem}{\bf Theorem}
\numberwithin{equation}{section}

\DeclareMathOperator{\ric}{Ric}
\allowdisplaybreaks[4]
\begin{document}

\title{A note on gradient Ricci soliton warped metrics}
\author[Jos\'e N.V. Gomes, Marcus A.M. Marrocos]{Jos\'e N.V. Gomes$^1$, Marcus A.M. Marrocos$^2$}
\author[Adrian V.C. Ribeiro]{Adrian V.C. Ribeiro$^3$}
\address{$^{1}$Departamento de Matem\'atica, Universidade Federal de S\~ao Carlos, S\~ao Carlos, S\~ao Paulo, Brazil.}
\address{$^2$CMCC-Universidade Federal do ABC, Santo Andr\'e, S\~ao Paulo, Brazil.}
\address{$^3$Escola Superior de Tecnologia, Universidade do Estado do Amazonas, Manaus, Amazonas, Brazil.}
\email{$^1$jnvgomes@ufscar.br; jnvgomes@pq.cnpq.br}
\email{$^2$marcus.marrocos@ufabc.edu.br}
\email{$^3$aribeiro@uea.edu.br}
\urladdr{$^1$https://www2.ufscar.br}
\urladdr{$^2$http://cmcc.ufabc.edu.br}
\urladdr{$^3$http://www3.uea.edu.br}

\begin{abstract}
In this note, we prove triviality and nonexistence results for gradient Ricci soliton warped metrics. The proofs stem from the construction of gradient Ricci solitons that are realized as warped products, from which we know that the base spaces of these products are Ricci-Hessian type manifolds. We study this latter class of manifolds as the most appropriate setting to prove our results.
\end{abstract}
\maketitle

\section{Introduction}
A complete Riemannian metric $g$ on a smooth manifold $M$ is a gradient Ricci soliton if there exists a smooth function $\phi$ on $M$ such that the Ricci tensor of $g$ is given by
\begin{equation*}
\ric+\nabla^2\phi=\lambda g,
\end{equation*}
for some constant $\lambda\in\Bbb{R}$. Note that the parameters in this equation are $g$ and $\phi$, while the constant $\lambda$ is obtained by taking trace of this equation. A gradient Ricci soliton is called {\it expanding}, {\it steady} or {\it shrinking} if $\lambda<0$, $\lambda=0$ or $\lambda>0$, respectively. When the \emph{potential function} $\phi$ is a constant, the pair $(M,g)$ is an Einstein manifold and $g$ is called a {\it trivial Ricci soliton}.

Given two Riemannian manifolds $(B^n,g_{B})$ and $(F^m,g_{F})$ as well as a positive smooth warping function $f$ on $B^n$, let us consider on the product manifold $B^n\times F^m$ the warped metric $g=\pi^{*}g_{B}+(f\circ \pi)^{2}\sigma^{*}g_{ F},$ where $\pi$ and $\sigma$ are the natural projections on $B^n$ and $F^m$, respectively. Under these conditions the product manifold is called the {\it warped product} of $B$ and $F$. If $f$ is a constant, then $(B^n\times F^m,g)$ is a standard Riemannian product.

Gradient Ricci solitons are self-similar solutions of the Ricci flow, and often arise as possible singularity models of this flow, see, e.g., Chow et al.~\cite{ChowEtal} or Hamilton~\cite{Hamilton}. Not only the literature about this subject is already very rich, but also it is not unlikely that the Ricci flow may play a fundamental role in the understanding of countless physical facts. For instance, the solution by Perelman of the celebrated Poincaré conjecture relies on Hamilton's careful study of the Ricci flow. Perelman did not himself publish this solution, but it can be checked in a series of preprints posted by him on arXiv.org in 2002 and 2003 or, alternatively, in a complete prove by Cao and Zhu~\cite{cao-zhu}.

Warped products appear in a natural manner in Riemannian geometry and their applications abound, see Bishop and O'Neill~\cite{BishopO'Neill} or O'Neill~\cite{oneill}. They are the source of examples and counter-examples of many classical geometric questions. For instance, warped products have been the key ingredient in the recent work of Marrocos and Gomes~\cite{Marrocos-Gomes} to give a partial answer to a question posed in 1990 by Steven Zelditch about the generic situation of the multiplicity of the eigenvalues of the Laplacian on principal bundles.

In this note we study gradient Ricci solitons that are realized as warped metrics $g$ on $M=B^n\times F^m$. For simplicity, we will say that $B^n\times_f F^m$ is a gradient Ricci soliton warped product and $g$ is a gradient Ricci soliton warped metric. By a result of Borges and Tenenblat we can assume without loss of generality that the potential function of such a soliton is the lift $\tilde{\varphi}$ of a smooth function $\varphi$ on $B^n$ to $M$, see Lemma~\ref{PotentialFunctionGRSWP}. Throughout this work $\psi$ stands for the smooth function $\psi=\varphi-m\ln f$ on $B$.  Note that in the setting of gradient Ricci soliton warped metrics the parameters are $\varphi$ and $f$. Hence, it is natural we assume conditions on these parameters to study nonexistence and triviality results of such a metrics. We will see that the class of Ricci-Hessian type manifolds $(B^n,g_B,\psi)$ is the most appropriate setting to prove our results.

Albeit the class of warped metrics with nonconstant warping functions provides a rich class of examples in Riemannian geometry, in this work, under natural geometric assumptions on the warping function as well as on the potential function of a gradient Ricci soliton warped metric, we give some obstructions for constructing such a metric.

The main geometric object of study of this note stems from the construction of gradient Ricci soliton warped products by Feitosa et al.~\cite{FFG}. We briefly describe it in Proposition~\ref{CharacterizationGRSPW}, from which we know that the base spaces of these products are Ricci-Hessian type manifolds.

We begin Section~\ref{Section2} with some comments and results on Ricci-Hessian type manifolds, which allow us to see that this class of manifolds is interesting in its own right. We use methods from weighted Riemannian manifold theory to prove the validity of a weak maximum principle at infinity for a specific diffusion operator on  Ricci-Hessian type manifolds, see Proposition~\ref{WMP}. In particular, such a principle is valid in the setting of gradient Ricci solitons as well as of $m$-quasi-Einstein manifolds, as proven before by Pigola et al.~\cite{Pigola et al.} and by Rimoldi~\cite{Rimoldi}, respectively.

To prove our theorems we restrict ourselves to Ricci-Hessian type manifolds as being the base spaces of gradient Ricci soliton warped products, which means that they must satisfy the additional Eq.~\eqref{SuffCondition2}, see Section~\ref{Section2.2}. The main theorem is a triviality result on this class of manifolds. Its proof is motivated by the corresponding result for gradient Ricci solitons proved in~\cite{Pigola et al.}.

\begin{theorem}\label{PropositionVarphi}
Let $(B^n,g_B,\psi)$ be a Ricci-Hessian type manifold with $\lambda<0$ satisfying Eq.~\eqref{SuffCondition2}. Then, the parameter function $\varphi$ is a constant provided it satisfies $|\nabla\varphi|\in L^p(B^n, e^{-\psi}d\mathrm{vol})$, for some $1\leq p\leq+\infty$.
\end{theorem}

The previous theorem is an improvement of Theorem~2 by Pigola et al. in the class of gradient Ricci soliton warped products $M=B^n\times_f F^m$. It is a consequence of the Borges and Tenenblat's result and Fubini's theorem. Indeed, since the potential function $\phi=\tilde\varphi$ and $d\mathrm{vol}_M=f^md\mathrm{vol}_{B\times F}$, we get
\begin{equation*} 
\int_M |\nabla\phi|^p e^{-\phi}d\mathrm{vol}_M = \int_Fd\mathrm{vol}_F\int_B|\nabla\varphi|^p e^{-\psi}d\mathrm{vol}_B,
\end{equation*}
which proves our statement as well as the naturalness of our hypothesis.

The next three theorems are applications of our results on the class of Ricci Hessian type manifolds satisfying Eq.~\eqref{SuffCondition2}.

\begin{theorem}\label{NonexistenceGRSPW}
Let $B^n\times_f F^m$ be a gradient steady Ricci soliton warped product with fiber having nonnegative scalar curvature. Then, it must be a standard Riemannian product with potential function satisfying $e^{-\varphi}\in L^1(B^n,d\mathrm{vol})$ provided  $f$ satisfies the following conditions: either $ f\in L^p(B^n,e^{-\psi}d\mathrm{vol})$, for some $1< p<+\infty$, or $ f\in L^1(B^n,e^{-\psi}d\mathrm{vol})$ and $f(x)=O(e^{\alpha r(x)^{2-\epsilon}})$ as  $r(x)\rightarrow +\infty$, for some constants $\alpha,\epsilon>0$.
\end{theorem}

Theorem~\ref{NonexistenceGRSPW} is in the spirit of Rimoldi~\cite{Rimoldi}. In this paper, the author proves triviality results for Einstein warped products with noncompact bases. His proofs rely on maximum principles at infinity and Liouville-type theorems on $m$-quasi-Einstein manifolds.

The next result is a nonexistence theorem for the expanding case.

\begin{theorem}\label{TrivialityGRSPW}
It is not possible to construct a gradient expanding Ricci soliton warped product $B^n\times_f F^m$ with fiber having nonnegative scalar curvature and warping function satisfying the following conditions: either $f\in L^p(B^n,e^{-\psi}d\mathrm{vol})$, for some $1<p\leq+\infty$, or $f\in L^1(B^n,e^{-\psi}d\mathrm{vol})$ and $ f(x)=O(e^{\alpha r(x)^{2-\epsilon}})$ as  $r(x)\rightarrow +\infty$, for some constants $\alpha,\epsilon>0$.
\end{theorem}

Notice that Theorem~~\ref{TrivialityGRSPW} is already an extension of Theorem 11 in~\cite{Rimoldi}. Moreover, Theorems~\ref{NonexistenceGRSPW} and \ref{TrivialityGRSPW} constitute an extension of Theorem~9 in~\cite{Rimoldi}.

It is well-known that Bryant constructed nontrivial examples of noncompact gradient steady Ricci soliton warped products with fiber having positive scalar curvature and with a warping function invariant by rotations. Bryant did not himself publish this result, but it can be checked in \cite{ChowEtal}. Moreover, Feitosa et al. constructed nontrivial examples of gradient expanding Ricci soliton warped products with fiber having nonpositive scalar curvature and with a warping function invariant under translation, see \cite[Corollary 2]{FFG}. We observe that these warping functions do not satisfy the hypotheses in Theorems~\ref{NonexistenceGRSPW} and \ref{TrivialityGRSPW}. Thus, the integrability assumptions in these theorems are necessary.

\begin{remark}
We can also consider in Theorems~\ref{NonexistenceGRSPW} and~\ref{TrivialityGRSPW} the weighted Riemannian volume density  $e^{-\varphi}d\mathrm{vol}$ instead of $e^{-\psi}d\mathrm{vol}$. Indeed, if $f\in L^p(B^n,e^{-\varphi}d\mathrm{vol})$ and $q=p-m>1$, then $\int_Bf^qe^{-\psi}d\mathrm{vol}=\int_Bf^{q+m}e^{-\varphi}d\mathrm{vol}<\infty$, i.e, $f\in L^q(B^n,e^{-\psi}d\mathrm{vol})$. For $p-m=1$ case, we must suppose in addition that $f(x)=O(e^{\alpha r(x)^{2-\epsilon}})$ as $r(x)\rightarrow +\infty$, for some constants $\alpha,\epsilon>0$. In both cases, the results are as described in these theorems.
\end{remark}

According to Corollary~1 in \cite{FFG}, the compactness of the base of a gradient Ricci soliton warped product implies restrictions on its existence. Besides, Proposition~4 in \cite{FFG} establishes a compactness criterion to Ricci-Hessian type manifold, which implies restrictions on the existence of gradient shrinking Ricci soliton warped products. Here, we also prove a nonexistence result in the shrinking case.

\begin{theorem}\label{TheoremLnf}
It is not possible to construct a gradient shrinking Ricci soliton warped product $B^n\times_fF^m$ with fiber having nonpositive scalar curvature and warping function satisfying the following conditions: either $|\nabla\ln f|\in L^p(B,e^{-\psi}d\mathrm{vol})$, for some $1<p\leq+\infty$, or $|\nabla\ln f|\in L^1(B^n,e^{-\psi}d\mathrm{vol})$ and $|\nabla\ln f|=O(e^{\alpha r(x)^{2-\epsilon}})$ as  $r(x)\rightarrow +\infty$, for some constants $\alpha,\epsilon>0$.
\end{theorem}

\section{Ricci-Hessian type manifolds}\label{Section2}
Let $(B^n,g_B,e^{-h}d\mathrm{vol})$ be an $n$-dimensional weighted manifold, where $h$ is a smooth function on $B$ and $d\mathrm{vol}$ is the Riemannian volume density on $(B^n,g_B)$. A natural extension of the Ricci tensor to weighted Riemannian manifolds is the {\it Bakry-Emery Ricci tensor}
\begin{equation*}
Ric_h=Ric+\nabla^2h.
\end{equation*}
A modification this equation is known as {\it $m$-Bakry-Emery Ricci tensor}, namely
\begin{equation*}
Ric^m_{h}=Ric+\nabla^2h-\frac{1}{m}dh\otimes dh,
\end{equation*}
for some $0<m\leq+\infty$. In particular, an $m$-quasi-Einstein metric $g_B$ satisfies $\ric^m_{h}=\lambda g_B$, for some constant $\lambda$. Besides, by introducing the function $f=e^{-\frac{h}{m}}$, with $m<+\infty$, one has
\begin{equation*}
Ric^m_{h}=Ric-\frac{m}{f}\nabla^2f.
\end{equation*}

It is known that the class of $m$-quasi-Einstein metrics is the most appropriate setting to study Einstein warped products, see, e.g., the results by Rimoldi~\cite{Rimoldi}.

Throughout this note the primary geometric object associated with our results is a complete Riemannian manifold $(B^n,g_B)$ with two smooth functions $f>0$ and $\varphi$ on $B$. We can further consider the weighted Riemannian manifold $(B^n,g_B,e^{-\psi}d\mathrm{vol})$, where $\psi$ is the smooth function on $B$ given by $\psi=\varphi-m\ln f$, for $0<m<+\infty$. In this case, we have
\begin{equation}\label{ModifiedRicci}
Ric_\psi=Ric+\nabla^2\varphi-\frac{m}{f}\nabla^2f+\frac{m}{f^2}df\otimes df.
\end{equation}

For our purposes, we define the following modification of the $m$-Bakry-Emery Ricci tensor
\begin{equation*}
Ric_{\varphi,h,m}:=Ric^m_h+\nabla^2\varphi.
\end{equation*}

An interesting case occurs when the Riemannian metric $g_B$ satisfies
\begin{equation*}
Ric_{\varphi,\xi,m}=\lambda g_B,
\end{equation*}
for some smooth functions $\xi=-m\ln f$ and $\lambda$ on $B$, or equivalently
\begin{align}\label{SuffCondition1}
Ric+\nabla^2\varphi-\frac{m}{f}\nabla^2f=\lambda g_B.
\end{align}

To give a motivation of where the study of this equation comes from, we refer to Maschler's work~\cite{Maschler} in which he was interested in conformal changes of Kähler-Ricci solitons in order to obtain new Kähler-metrics. To do this, he introduced the notion of {\it Ricci-Hessian equation}, namely
\begin{equation*}
\ric+\alpha\nabla^2h=\beta g,
\end{equation*}
where $\alpha$ and $\beta$ are smooth functions. Feitosa et al.~\cite[Remark~1]{Feitosa et al.} showed how Eq.~\eqref{SuffCondition1} can be reduced to a Ricci-Hessian equation. Example~1 in~\cite{FFG} shows that the standard sphere and the hyperbolic space both satisfy Eq.~\eqref{SuffCondition1} for nontrivial functions. The same authors constructed a nontrivial example of \eqref{SuffCondition1} on $\Bbb{R}^n$ with parameter functions invariant under the action of an $(n-1)$-dimensional translation group, see Corollary~1 in~\cite{Feitosa et al.}.

Our main motivation to study Eq.~\eqref{SuffCondition1} is that the base spaces of gradient Ricci soliton warped products satisfy the referred equation for a  positive integer $m$ and a constant $\lambda$. To see this, we briefly describe the construction by Feitosa et al.~\cite{FFG}. They considered a complete Riemannian manifold $(B^n,g_B)$ with two parameters smooth functions $f>0$ and $\varphi$ satisfying Eq.~\eqref{SuffCondition1} and
\begin{equation}\label{SuffCondition2}
2\lambda\varphi-|\nabla\varphi|^2+\Delta\varphi+\frac{m}{f}\nabla\varphi(f)=c,
\end{equation}
for some constants $\lambda, m, c\in\Bbb{R}$, with $m\neq0$. By \cite[Proposition~3]{FFG}, the functions $f$ and $\varphi$ satisfy
\begin{equation}\label{SuffCondition3}
\lambda f^2+f\Delta f+(m-1)|\nabla f|^2-f\nabla\varphi(f)=\mu,
\end{equation}
for some constant $\mu\in\mathbb{R}$. In summary, we will need the following result which is a compilation of Proposition~2 and Theorem~3 in~\cite{FFG}.

\begin{proposition}[Feitosa et al.~\cite{FFG}]\label{CharacterizationGRSPW}
Let $M=B^n\times_f F^m$ be a gradient Ricci soliton warped product with potential function $\tilde{\varphi}$. Then, Eqs.~\eqref{SuffCondition1} and \eqref{SuffCondition2} hold on $B$ and the fiber $F$ is an Einstein manifold with Ricci tensor $\ric_F=\mu g_F$, where $\mu$ is given by Eq.~\eqref{SuffCondition3}. Conversely, let $B$ be a complete Riemannian manifold with two smooth functions $\varphi$ and $f>0$ satisfying Eqs.~\eqref{SuffCondition1} and \eqref{SuffCondition2}, for any $\lambda\in\Bbb{R}$. Take the constant $\mu$ given by Eq.~\eqref{SuffCondition3} and a complete Riemannian manifold $F$ of dimension $m$ and Ricci tensor $\ric_F=\mu g_F$. Then, we can construct a gradient Ricci soliton warped product $B^n\times_f F^m$ with potential function $\tilde{\varphi}$.
\end{proposition}

We now define precisely the main geometric object of study in this work, which stems from Proposition~\ref{CharacterizationGRSPW}.

A {\it Ricci-Hessian type manifold} is a complete Riemannian manifold $(B^n,g_B)$ with two smooth functions $f>0$ and $\varphi$ satisfying Eq.~\eqref{SuffCondition1}, for a positive integer $m$ and a constant $\lambda$.

For simplicity, we will say that $(B^n,g_B,\psi)$ is a Ricci-Hessian type manifold, where $\psi=\varphi-m\ln f$. Moreover, we will use the weighted Riemannian volume density $e^{-\psi}d\mathrm{vol}$ on $(B^n,g_B,\psi)$ whenever it is needed.

We continue by fixing notation and making comments about facts that will be used henceforth.

One standard way to address problems such as those proposed in our introduction is by using three classical tools in Riemannian geometry. The first tool we write in Lemma~\ref{Bochner}, which provides a Bochner-Weitzenböck type formula in terms of suitable geometric tensor on $(B^n,g_B,\psi)$. The traditional Bochner-Weitzenböck formula allows us to get a useful identity in the gradient Ricci soliton theory, namely
\begin{equation}\label{BochnerRich}
\frac{1}{2}\Delta_\psi|\nabla u|^2=|\nabla^2u|^2+\langle\nabla u,\nabla\Delta_\psi u\rangle+\ric_\psi(\nabla u,\nabla u),
\end{equation}
where $\Delta_\psi u:= \Delta u-\langle\nabla\psi,\nabla u\rangle$ is called  $\psi$-Laplacian. The second tool we state in Proposition~\ref{WMP}, which provides a {\it weak maximum principle at infinity} for a $C^2$-function $u:B\to\Bbb R$ satisfying $\sup_Bu=u^{\ast}<+\infty$. The validity of this principle ensures that there exists a sequence $\{x_k\}\subset B$ satisfying, for all $k\in\mathbb{N}$:
\begin{equation*}
(i)\ u(x_k)\geq u^{\ast}-\frac{1}{k} \quad \mbox{and} \quad (ii)\ (\Delta_\psi u)(x_k)\leq\frac{1}{k}.
\end{equation*}

The third tool is a $L^p$-Liouville type result associated to suitable drifting function. Here we will need the following version described below for the $\psi$-Laplacian. This result is an immediate consequence of \cite[Theorem~14 and Remark~15]{Pigola et al.}.

\begin{proposition}\label{Liouville}
Let $(B^n,g_B,\psi)$ be a Ricci-Hessian type manifold, and let us consider a function $u\in Lip_{loc}(B^n)$. If $u\in L^p(B^n,e^{-\psi}d\mathrm{vol})$, for some $1<p<+\infty$, and $\Delta_{\psi}u\geq0$ in the weak sense, then either $u$ is a constant or $u\leq 0$.
\end{proposition}

To establish the first two tools, we use the very convenient Eq.~\eqref{ModifiedRicci}, which is equivalent to
\begin{equation}\label{RelationRicci}
\ric_{\psi}=\ric_{\varphi,\xi,m}+\frac{m}{f^2}df\otimes df.
\end{equation}

As an immediate consequence of Eqs.~\eqref{BochnerRich} and \eqref{RelationRicci}, we obtain the required Bochner-Weitzenböck type formula associated with $\ric_{\varphi,\xi,m}$.

\begin{lemma}\label{Bochner}
Let $(B^n,g_B,\psi)$ be a Ricci-Hessian type manifold. For any smooth function $u$ on $B$ the following holds:
\begin{equation*}
\frac{1}{2}\Delta_{\psi}|\nabla u|^2=|\nabla^2u|^2+\langle\nabla u,\nabla\Delta_{\psi}u\rangle+\ric_{\varphi,\xi,m}(\nabla u,\nabla u)+\frac{m}{f^2}\langle\nabla f,\nabla u\rangle^2.
\end{equation*}
\end{lemma}

From Eq.~\eqref{RelationRicci} we also obtain the validity of the weak maximum principle at infinity for $\Delta_{\psi}$ on Ricci-Hessian type manifolds $(B^n,g_B,\psi)$. Indeed, if we assume $\ric_{\varphi,\xi,m}=\lambda g_B$, then $\ric_{\psi}\geq\lambda g_B$. Thus, by using \cite[Theorem~4.1]{WeiWylie} and \cite[Theorem~9]{Pigola et al.} one readily has the following proposition.

\begin{proposition}\label{WMP}
The weak maximum principle at infinity for the $\psi$-Laplacian holds on a Ricci-Hessian type manifold $(B^n,g_B,\psi)$.
\end{proposition}

More generally, we note that the previous weak maximum principle at infinity holds for weighted Riemannian manifolds $(B^n,g_B,e^{-\psi}d\mathrm{vol})$ satisfying $\ric_{\varphi,\xi,m}\geq\lambda g_B$.

\section{Ricci-Hessian type manifolds satisfying Eq.~\eqref{SuffCondition2}}\label{Section2.2}

In this section, we elaborate the proofs of the main theorems of this note. First of all, we restrict ourselves to Ricci-Hessian type metrics $g_B$ on $B$ as being a necessary condition to construct gradient Ricci soliton warped metrics, which means that they must satisfy the additional Eq.~\eqref{SuffCondition2} on $(B,g_B)$. In particular, when $f$ is a constant, a Ricci-Hessian type metric reduces to a gradient Ricci soliton which has been carefully studied by Hamilton~\cite{Hamilton}, in this case, Eq.~\eqref{SuffCondition2} is known as a Hamilton's equation, and Eq.~\eqref{SuffCondition3} is trivially satisfied. Taking $\varphi$ to be constant instead of $f$, we have an $m$-quasi Einstein metric, which is originates from the study of Einstein warped metrics, see Kim and Kim~\cite{DKim-YKim} or Besse~\cite{besse}, from which we know that Eq.~\eqref{SuffCondition3} is satisfied, while Eq.~\eqref{SuffCondition2} is trivially satisfied. Thus, this additional assumption on the Ricci-Hessian type manifolds is natural and very convenient.

We start by using Eqs.~\eqref{SuffCondition1}-\eqref{SuffCondition3} to obtain the following lemma.
\begin{lemma}\label{RewrittenEquations}
Let $(B^n,g_B,\psi)$ be a Ricci-Hessian type manifold satisfying Eq.~\eqref{SuffCondition2}. Then,
\begin{align*}
&\Delta\psi=n\lambda-S+m|\nabla\ln f|^2,\\
&\Delta_{\psi}\varphi=c-2\lambda\varphi,\\
&\Delta_{\psi}\ln f=\frac{1}{f^2}(\mu-\lambda f^2).
\end{align*}
\end{lemma}

We are now in position to prove the first result in our study of triviality and nonexistence of gradient Ricci soliton warped products.

\begin{proposition}\label{PreliminaryResults}
Let $(B^n,g_B,\psi)$ be a Ricci-Hessian type manifold satisfying Eq.~\eqref{SuffCondition2}. Suppose $\lambda\leq0$, the following holds:
\begin{enumerate}
\item\label{PreliminaryResults-a} If $\varphi\in L^{p}(B^n,e^{-\psi}d\mathrm{vol})$, for some $1<p<+\infty$, then $\varphi$ does not change sign on $B^n$.
\item\label{PreliminaryResults-b} For $\lambda=0$ and $\mu\geq0$, if the function $f$  satisfies the following conditions: either $ f\in L^p(B^n,e^{-\psi}d\mathrm{vol})$, for some $1< p<+\infty$, or $ f\in L^1(B^n,e^{-\psi}d\mathrm{vol})$ and $f(x)=O(e^{\alpha r(x)^{2-\epsilon}})$ as  $r(x)\rightarrow +\infty$, for some constants $\alpha,\epsilon>0$, then $f$ must be constant and $e^{-\varphi}\in L^1(B,d\mathrm{vol})$.
\item\label{PreliminaryResults-c} For $\lambda<0$ and $\mu\geq0$, there is no such a Ricci-Hessian type manifold provided that $f$ satisfies the following conditions: either  $f\in L^p(B^n,e^{-\psi}d\mathrm{vol})$, for some $1<p\leq+\infty$, or $f\in L^1(B^n,e^{-\psi}d\mathrm{vol})$ and $f(x)=O(e^{\alpha r(x)^{2-\epsilon}})$ as  $r(x)\rightarrow +\infty$, for some constants $\alpha,\epsilon>0$.
\end{enumerate}
\end{proposition}

\begin{proof}
\noindent \textbf{Part~\eqref{PreliminaryResults-a}:} First we assume $c\geq 0$. In this case, we consider the locally lipschitz function $\varphi_+=\max\{\varphi,0\}$ which  satisfies $\Delta_{\psi}\varphi_+=c-2\lambda\varphi_+\geq 0$ (see Lemma~\ref{RewrittenEquations}), since $\lambda\leq0$. From Proposition~\ref{Liouville}, it must be constant, thus, if there exists a point $x_0\in B$ such that $\varphi(x_0)>0$, then $\varphi\equiv\varphi(x_0)>0$. Otherwise, we have $\varphi\leq0$ on $B$. For  $c<0$ case, we apply an analogous argument to $\varphi_-=\min\{\varphi,0\}$ to conclude that either $\varphi\equiv\varphi(x_0)<0$ or $\varphi\geq0$ on $B$.
	
To show part~\eqref{PreliminaryResults-b} and~\eqref{PreliminaryResults-c} we use Lemma~\ref{RewrittenEquations} to obtain
\begin{equation}\label{mu&Lambda}
f\Delta_{\psi}f=\mu-\lambda f^2+|\nabla f|^2.
\end{equation}

\vspace{0.1cm}
\noindent \textbf{Part~\eqref{PreliminaryResults-b}:} For the $L^p$ case, we have $f\Delta_{\psi}f\geq0$ and then $f$ must be constant by Proposition~\ref{Liouville}. For the $L^1$ case, the result is a consequence of~\cite[Theorem 17]{Pigola et al.}. In both cases, $f^{p+m}\int_Be^{-\varphi}d\mathrm{vol}=\int_B f^p e^{-\psi}d\mathrm{vol}<\infty$, thus $e^{-\varphi}\in L^1(B,d\mathrm{vol})$. In particular, if $\varphi$ is constant, then $B$ must be of finite volume.
	
\vspace{0.1cm}
\noindent \textbf{Part~\eqref{PreliminaryResults-c}:} First we note that $f$ must be nonconstant from Eq.~\eqref{mu&Lambda}. Assume that $f\in L^\infty(B)$, i.e., $f^{\ast}=\sup_{B}f<+\infty$. By Proposition~\ref{WMP}, there exists a sequence $\{x_k\}\subset B$ along which
\begin{equation*}
\lim_{k\to+\infty}f(x_k)=f^*\quad \mbox{and} \quad \limsup_{k\to+\infty}(\Delta_{\psi}f)(x_k)\leq 0 .
\end{equation*}
On the other hand, we get from Eq.~\eqref{mu&Lambda} that $\Delta_{\psi}f\geq-\lambda f>0$, which implies a contradiction, since $\lambda<0$ and $f>0$.

For the $L^p$ case, we obtain that $\Delta_{\psi}f>0$, since $\lambda<0$, and thus we apply Proposition~\ref{Liouville} to conclude that $f$ must be constant, which is again a contradiction. For the $L^1$ case, the result is a consequence of~\cite[Theorem 17]{Pigola et al.}.
\end{proof}

As a consequence we take $f$ to be constant in a Ricci-Hessian type manifold and by using part~\eqref{PreliminaryResults-a} of Proposition~\ref{PreliminaryResults}, we obtain the following result.

\begin{corollary}
Let $g$ be a gradient expanding or steady Ricci soliton on a smooth manifold $B^n$ with potential function $\varphi$. If $\varphi\in L^{p}(B^n,e^{-\varphi}d\mathrm{vol})$, for some $1<p<+\infty$, then $\varphi$ does not change sign on $B^n.$
\end{corollary}

We now combine Lemmas~\ref{Bochner} and~\ref{RewrittenEquations} to obtain the following identities.
\begin{lemma}\label{Bochner2}
Let $(B^n,g_B,\psi)$ be a Ricci-Hessian type manifold satisfying Eq.~\eqref{SuffCondition2}. The following hold:
\begin{enumerate}
\item\label{Lem-Bochner2a} $\frac{1}{2}\Delta_\psi|\nabla\varphi|^2=|\nabla^2\varphi|^2-\lambda|\nabla\varphi|^2+\frac{m}{f^2}\langle\nabla\varphi,\nabla f\rangle^2$.
\item\label{Lem-Bochner2b} $\frac{1}{2}\Delta_\psi|\nabla\ln f|^2=|\nabla^2\ln f|^2-\frac{2\mu}{f^2}|\nabla\ln f|^2+\lambda|\nabla\ln f|^2+\frac{m}{f^2}\langle\nabla\ln f,\nabla f\rangle^2$.
\end{enumerate}
\end{lemma}

We will also need the following inequalities, which follow from Lemma~\ref{Bochner2} and a well-known Kato's inequality.
\begin{lemma}\label{bochner ineq}
Let $(B^n,g_B,\psi)$ be a Ricci-Hessian type manifold satisfying Eq.~\eqref{SuffCondition2}. The following inequalities hold in the weak sense:
\begin{equation}\label{Inequality1}
|\nabla\varphi|\Delta_\psi|\nabla\varphi|\geq-\lambda|\nabla\varphi|^2+\frac{m}{f^2}\langle\nabla\varphi,\nabla f\rangle^2
\end{equation}
and
\begin{equation}\label{Inequality2}
|\nabla\ln f|\Delta_\psi|\nabla\ln f|\geq \lambda|\nabla\ln f|^2-\frac{2\mu}{f^2}|\nabla\ln f|^2+\frac{m}{f^2}\langle\nabla\ln f,\nabla f\rangle^2.
\end{equation}
\end{lemma}

\begin{proof}
For a smooth function $u$ on $B$ we have
\begin{equation*}
\frac{1}{2}\Delta_\psi|\nabla u|^2=|\nabla u|\Delta_\psi|\nabla u|+|\nabla|\nabla u||^2.
\end{equation*}
Combining this latter equality with the Kato's inequality
\begin{equation*}
|\nabla^2u|^2\geq|\nabla|\nabla u||^2,
\end{equation*}
one has
\begin{equation*}
|\nabla u|\Delta_\psi|\nabla u|\geq \frac{1}{2}\Delta_\psi|\nabla u|^2-|\nabla^2 u|^2.
\end{equation*}
Taking $u=\varphi$ into the latter inequality and using part~\eqref{Lem-Bochner2a} of Lemma~\ref{Bochner2} we have
\begin{equation*}
|\nabla\varphi|\Delta_\psi|\nabla\varphi|\geq\frac{1}{2}\Delta_\psi|\nabla\varphi|^2-|\nabla^2\varphi|^2= -\lambda|\nabla\varphi|^2+\frac{m}{f^2}\langle\nabla\varphi,\nabla f\rangle^2.
\end{equation*}
The second required inequality is analogously obtained.
\end{proof}

In what follows we continue presenting some results for Ricci-Hessian type manifolds satisfying Eq.~\eqref{SuffCondition2}. Theorem~\ref{PropositionVarphi} gives the conditions for Ricci-Hessian type manifold to be an $m$-quasi-Einstein metric. Before proving Theorem~\ref{PropositionVarphi} let us establish the following lemma.

\begin{lemma}\label{estimate potential}
Let $(B^n,g_B,\psi)$ be a Ricci-Hessian type manifold with $\lambda<0$ satisfying Eq.~\eqref{SuffCondition2}. Then, having fixed an origin $o\in B^n$, there exists a constant $c> 0$ such that $|\nabla\varphi|\leq c(1+r(x))$, where $r(x)=dist_{(B^n,g_B)}(x,o)$ is the distance function.
\end{lemma}
\begin{proof}
Let us consider $\mu$ given by~\eqref{SuffCondition3} and an $m$-dimensional space form $(F^m,g_F)$ with constant scalar curvature $S_F=m\mu$ such that the warped metric $g=g_B+f^2 g_F$ is a gradient expanding Ricci soliton on $B^n\times F^m $ with potential function $\tilde{\varphi}$. Now we take a point $(o,q)\in B^n\times_f F^m$ so that $\overline{r}(x,y)=dist_{(B^n\times_f F^m)}((x,y),(o,q))$ is the associated distance function. It is easy to see that $\overline{r}(x,q)=r(x)$, where $(x,q)\in B^n\times\{q\}$. It follows from Zhang~\cite{Zhang} that there exists a constant $c>0$ such that $|\nabla\varphi(x)|=|\nabla\tilde{\varphi}(x,q)|\leq c(1+\overline{r}(x,q))=c(1+{r}(x))$.
\end{proof}

With the previous results in mind we prove our main theorem.

\subsection{Proof of Theorem~\ref{PropositionVarphi}}
\begin{proof}
Since $\lambda<0$, by inequality~\eqref{Inequality1} of Lemma~\ref{bochner ineq} we have
\begin{equation*}
\frac{1}{2}\Delta_\psi|\nabla \varphi|^2=|\nabla \varphi|\Delta_\psi|\nabla \varphi|+|\nabla|\nabla \varphi||^2\geq-\lambda|\nabla\varphi|^2+|\nabla|\nabla \varphi||^2\geq0.
\end{equation*}
Assume that $|\nabla\varphi|\in L^{\infty}(B)$, by Proposition~\ref{WMP} there exists a sequence $\{x_k\}\subset B$ such that
\begin{equation*}
\limsup_{k\to+\infty}\ (\Delta_{\psi}|\nabla\varphi|^2)(x_k)\leq0 \quad \mbox{and} \quad \lim_{k\to+\infty}|\nabla\varphi|^2(x_k)=\sup_{B}|\nabla\varphi|^2.
\end{equation*}
Hence, we obtain that $\sup_{B}|\nabla\varphi|^2=0$ and then $\varphi$ is a constant.
	
Assume now that $|\nabla\varphi|\in L^p(B,e^{-\psi}d\mathrm{vol})$, for some $1<p<+\infty$. Again by \eqref{Inequality1} we get $|\nabla\varphi|\Delta_{\psi}|\nabla\varphi|\geq0$ and since $|\nabla\varphi|\in Lip_{loc}(B^n)$, then we can apply Proposition~\ref{Liouville} to obtain that $|\nabla\varphi|$ is a constant. Thus, from part~\eqref{Lem-Bochner2a} of Lemma~\ref{Bochner2}, $\varphi$ must be constant.

For the $L^1$ case, we have by Lemma~\ref{estimate potential} that $|\nabla\varphi|\leq c(1+r(x))$. Moreover, since $|\nabla\varphi|\Delta_\psi|\nabla\varphi|\geq0$ and $|\nabla\varphi|\in Lip_{loc}(B^n)$, we are in position to apply the result of \cite[Theorem 17]{Pigola et al.} to obtain that $\varphi$ is a constant.
\end{proof}

Taking $f$ to be constant in Theorem~\ref{PropositionVarphi}, we recover the following result.

\begin{corollary}[Pigola et al.~\cite{Pigola et al.}]\label{TheoremVarphi}
Let $g$ be a gradient expanding Ricci soliton on a smooth manifold $B^n$ with potential function $\varphi$. Then, $g$ is a trivial Ricci soliton provided $|\nabla\varphi|$ lies in $L^p(B^n, e^{-\varphi}d\mathrm{vol})$, for some $1\leq p\leq+\infty$.
\end{corollary}

Next, we prove the following nonexistence result.

\begin{proposition}\label{Propositionlnf}
There is no Ricci-Hessian type manifold $(B^n,g_B,\psi)$ satisfying Eq.~\eqref{SuffCondition2} with $\lambda>0$ and $\mu\leq0$ provided that the parameter function $f$ satisfies the following conditions: either $|\nabla\ln f|\in L^p(B,e^{-\psi}d\mathrm{vol})$, for some $1<p\leq+\infty$, or $|\nabla\ln f|\in L^1(B^n,e^{-\psi}d\mathrm{vol})$ and $|\nabla\ln f|=O(e^{\alpha r(x)^{2-\epsilon}})$ as  $r(x)\rightarrow +\infty$, for some constants $\alpha,\epsilon>0$.
\end{proposition}

\begin{proof}
First note that, by Lemma~\ref{RewrittenEquations}, the function $f$ must be nonconstant, since $\lambda>0$ and $\mu\leq0$. For the $L^\infty$ case, we use Proposition~\ref{WMP} to ensure that there exists a sequence $\{x_k\}\subset B$ such that
\begin{equation*}
\limsup_{k\to+\infty}\ (\Delta_{\psi}|\nabla\ln f|^2)(x_k)\leq0 \quad \mbox{and} \quad \lim_{k\to+\infty}|\nabla\ln f|^2(x_k)=\sup_{B}|\nabla\ln f|^2.
\end{equation*}
On the other hand, from part~\eqref{Lem-Bochner2b} of Lemma~\ref{Bochner2} we have
\begin{equation*}
\Delta_{\psi}|\nabla\ln f|^2\geq2\lambda|\nabla\ln f|^2\geq0.
\end{equation*}
Hence, we obtain that $\lambda\sup_{B}|\nabla\ln f|^2=0$, which implies in a contradiction.
	
For the $L^p$ case, by inequality~\eqref{Inequality2} we obtain that $|\nabla\ln f|\Delta_\psi|\nabla\ln f|\geq0$, since $\lambda>0$ and $\mu\geq0$. Thus, by Proposition~\ref{Liouville}, $|\nabla\ln f|$ must be constant, which implies again a contradiction.

For the $L^1$ case,  the result is a  consequence of~\cite[Theorem 17]{Pigola et al.}.
\end{proof}

\section{Proof of Theorems~\ref{NonexistenceGRSPW}, \ref{TrivialityGRSPW} and \ref{TheoremLnf} }\label{proofmr}

Let $B^n\times_fF^m$ be a gradient Ricci soliton warped product. By Borges and Tenenblat~\cite{TenenblatBorges}, we can assume without loss of generality that the potential function is the lift $\tilde{\varphi}=\varphi\circ\pi$ of a smooth function $\varphi$ on $B$ to product. By Proposition~\ref{CharacterizationGRSPW}, the base space $B$ is a Ricci-Hessian type manifold satisfying \eqref{SuffCondition1} and \eqref{SuffCondition2}, while the fiber $F$ is an Einstein manifold with Ricci tensor $\ric_F=\mu g_F$, where $\mu$ is given by \eqref{SuffCondition3}.

For the sake of clarity of exposition and due to the importance of the result itself, we give now a proof of the Borges and Tenenblat's result.

\begin{lemma}[Borges and Tenenblat~\cite{TenenblatBorges}]\label{PotentialFunctionGRSWP}
Let $B^n\times_fF^m$ be a warped product with metric $g=\pi^*g_B+(f\circ\pi)^2\sigma^*g_F$ and warping function $f$. If $g$ is a gradient Ricci soliton, then the potential function depends only on the base.
\end{lemma}

\begin{proof}
When $f$ is a constant, $(B\times F,g)$ is a standard Riemannian product and the result of the lemma is well-known. Indeed, this case follows by simple analysis on Eq.~\eqref{RicciSolitonEquation}.

For $f$ nonconstant the proof follows the same argument as the more general case in~\cite{TenenblatBorges}. Suppose $g$ is a gradient Ricci soliton warped metric with potential function $\phi$. Then, the Ricci tensor of $g$ satisfies
\begin{equation}\label{RicciSolitonEquation}
\ric+\nabla^2\phi=\lambda g,
\end{equation}
for some constant $\lambda\in\Bbb R$. By the standard expressions for the Ricci tensor and for the Hessian of a smooth function $\phi$ on the warped product (see Bishop and O'Neill~\cite{BishopO'Neill}), we obtain
\begin{equation*}
0=(\nabla^2\phi)(X,U)=X(U(\phi))-(\nabla_XU)(\phi)=X(U(\phi))-\frac{X(f)}{f}U(\phi),
\end{equation*}
for all $X,U$ horizontal and vertical vector fields, respectively. Next, we compute
\begin{equation*}
X(U(\phi f^{-1}))=X(U(\phi)f^{-1})=\frac{1}{f}\left[X(U(\phi))-\frac{X(f)}{f}U(\phi)\right]=0.
\end{equation*}
Hence, the function $U(\phi f^{-1})$ depends only on the fiber $F$. Thus, without loss of generality, we can assume that $\phi=\varphi+fh,$ for some functions $\varphi$ on $B$ and $h$ on $F$. Now we take a unitary geodesic $\gamma$ on $(B^n,g_B)$, so that equation~\eqref{RicciSolitonEquation} reads along $\gamma$ as
\begin{equation*}
\ric_B(\gamma',\gamma')+\varphi''-\frac{m}{f}f''+h f''=\lambda.
\end{equation*}
Thus, for any vertical vector field $V$, one has $f''V(h)=0$. Since we are considering $f$ positive and nonconstant, there exists a point $(p,q)\in B^n\times F^m$ such that $f''(p)\neq0$, for all $q\in F$. Then, $h$ must be constant on $F$, i.e., $\phi$ depends only on $B$.
\end{proof}

We are ready to give the proofs.

\subsection{Proof of Theorem~\ref{NonexistenceGRSPW}}
\begin{proof}
It is immediate consequence of Proposition~\ref{CharacterizationGRSPW} and part~\eqref{PreliminaryResults-b} of Proposition~\ref{PreliminaryResults}.
\end{proof}

\subsection{Proof of Theorem~\ref{TrivialityGRSPW}}
\begin{proof}
It is immediate consequence of Proposition~\ref{CharacterizationGRSPW} and part~\eqref{PreliminaryResults-c} of Proposition~\ref{PreliminaryResults}.
\end{proof}

\subsection{Proof of Theorem~\ref{TheoremLnf}}
\begin{proof}
It is immediate consequence of Propositions~\ref{CharacterizationGRSPW} and \ref{Propositionlnf}.
\end{proof}

\section{Concluding remarks}

Under the same hypotheses of Theorem~\ref{TheoremLnf}, we prove a triviality result in the class of gradient steady Ricci soliton warped products with fiber having nonpositive scalar curvature, namely:

\begin{theorem}
Let $B^n\times_f F^m$ be a gradient steady Ricci soliton warped product with fiber having nonpositive scalar curvature. Then, it must be a standard Riemannian product provided $f$ satisfies the following conditions: either $|\nabla\ln f|\in L^p(B,e^{-\psi}d\mathrm{vol})$, for some $1<p\leq+\infty$, or $|\nabla\ln f|\in L^1(B^n,e^{-\psi}d\mathrm{vol})$ and $|\nabla\ln f|=O(e^{\alpha r(x)^{2-\epsilon}})$ as  $r(x)\rightarrow +\infty$, for some constants $\alpha,\epsilon>0$.
\end{theorem}

The prove is an immediate consequence of Proposition~\ref{CharacterizationGRSPW} together with the following proposition.

\begin{proposition}\label{Propositon_steady_case_triviality}
Let $(B^n,g_B,\psi)$ be a Ricci-Hessian type manifold satisfying Eq.~\eqref{SuffCondition2} with $\lambda=0$ and $\mu\leq0$. Then, the parameter function $f$ is a constant provided it satisfies the following conditions: either $|\nabla\ln f|\in L^p(B,e^{-\psi}d\mathrm{vol})$, for some $1<p\leq+\infty$, or $|\nabla\ln f|\in L^1(B^n,e^{-\psi}d\mathrm{vol})$ and $|\nabla\ln f|=O(e^{\alpha r(x)^{2-\epsilon}})$ as  $r(x)\rightarrow +\infty$, for some constants $\alpha,\epsilon>0$.
\end{proposition}

\begin{proof}
Following the same steps as in the proof of Proposition~\ref{Propositionlnf}, we conclude that $|\nabla\ln f|$ is constant. Using part~\eqref{Lem-Bochner2b} of Lemma~\ref{Bochner2} if necessary, we get $|\nabla\ln f|=0$. Hence, $f$ must be constant.
\end{proof}

\section{Acknowledgements}
We would like to express our sincere thanks to the anonymous referee for his/her careful reading and useful comments which helped us improve our paper. This work has been partially supported by Conselho Nacional de Desenvolvimento Cient\'ifico e Tecnol\'ogico (CNPq), of the Ministry of Science, Technology and Innovation of Brazil (Grants 428299/2018-0 and 307374/2018-1).

\end{document}